\documentclass[a4paper,12pt]{article}
\usepackage{amsmath,amsthm,hyperref}
\author{Peter Vassilev\\
Institute of Biophysics and Biomedical Engineering\\
Bulgarian Academy of Sciences\\
Acad. G. Bonchev Str., Bl. 105, Sofia-1113, Bulgaria\\
e-mail:\texttt{peter.vassilev@gmail.com}
}
\title{A note on an inequality involving the sides and medians in a triangle}
\date{}

\newtheorem{theorem}{Theorem}
\newtheorem{lemma}{Lemma}
\newtheorem{corollary}{Corollary}
\newtheorem{remark}{Remark}
\newtheorem{openproblem}{Open Problem}

\begin{document}
\maketitle

\hfill To the memory of my grandmothers \\
\mbox{}\hfill Zorka Misanova (1918-2016) \\
\mbox{}\hfill Olga Vasileva (1925-2017)

\abstract{The main focus of the present paper is the following inequality 
\[\left( \sqrt{bc}-a\right) m_a+ \left(\sqrt{ac}-b\right)m_b+\left(\sqrt{ab}-c\right)m_c \geq 0,\] where $a,b,c$ are the sides of a non-degenerate triangle and $m_a,m_b,m_c,$ the respective medians; which was conjectured to be true but had not been proved. We provide a proof. We also show that analogous inequality is true when the medians are replaced by the altitudes or the internal angle bisectors. Finally, we conclude with an open problem regarding the Cevians which would satisfy such inequality.}

\section{Introduction}\label{intro} In the 1980s the Bulgarian mathematician Georgi Boychev observed that the following inequality seems to hold for an arbitrary triangle with sides $a,b,c$ and medians $m_a,m_b,m_c$:
\begin{equation}\label{INEQ}
 \sqrt{bc} m_a+ \sqrt{ac} m_b+\sqrt{ab} m_c \geq a m_a+ b m_b+c m_c 
\end{equation}
He made numerous numerical experiments which confirmed the conjectured relation, however, he was unable to find a definitive proof\footnote{We challenge the interested readers to test their ideas on possible proof before reading further}. He had communicated his conjecture in personal conversations with various colleagues including Mladen Vassilev-Missana, who relayed this story to us. Perhaps because such ``Olympiad type problems'' are not of great interest to the professional mathematicians or due to the fact that it was not generally known to puzzle solvers, not many people undertook the task to find a proof.

Whatever the case may be, the problem, as far as we know, had remained unresolved to this day. 

Before setting on finding a proof we made an investigation regarding similar problems. We have found the following inequality proposed by Kee-Wai Lau~\cite{Crux}:
\begin{equation}\label{INEQS}
\left({bc}-a^2\right) m_a+ \left({ac}-b^2\right)m_b+\left({ab}-c^2\right)m_c \geq 0
\end{equation}
which looks similar to \eqref{INEQ}. Its proof relies on the following inequality:
\begin{equation}\label{INEQKEY}
\begin{cases}
2a m_a \leq c m_b+b m_c\\
2b m_b \leq a m_c+c m_a\\
2c m_c \leq a m_b+b m_a
\end{cases}
\end{equation}
A purely algebraic proof is given by Marcin E. Kuczma in \cite{Cruxsol}. The same inequality is given by Dorin Andrica in \cite{MathematicalReflections} where geometrical proof is  given. For other geometrical proofs see also \cite{cuttheknot}. 

As it turns out \eqref{INEQKEY} is also crucial for the proof of~\eqref{INEQ}. In order to establish the proof we require some additional observations and finally the application of arithmetic mean - geometric mean inequality\footnote{In fact we only need a very particular case of it -- namely that if the product of three positive numbers equals one, their sum is greater or equal to three.}.

\section{Preliminary results}

\begin{lemma}\label{isosceles} 
Inequality \eqref{INEQ} is true for any isosceles triangle.
\end{lemma}
\begin{proof} First we observe that \eqref{INEQ} turns into identity for an equilateral triangle. 

Thus, without loss of generality we can have either: 1)$a=b<c$ or 2)$a<b=c.$

Let us consider case 1). Then \eqref{INEQ} becomes
\begin{equation}\label{INEQFIRSTISOS}
2\sqrt{ac} m_a+a m_c \geq 2 a m_a+c m_c 
\end{equation}
Dividing \eqref{INEQFIRSTISOS} by $c^2,$ and putting $x=\frac{a}{c}<1,$ it is equivalent to:
\begin{align*}
2 \sqrt{x} \sqrt{2+x^2}+x\sqrt{4x^2-1} - 2x \sqrt{2+x^2}- \sqrt{4x^2-1} \\
= 2\sqrt{x}(1-\sqrt{x})\sqrt{2+x^2}+(\sqrt{x}-1)(\sqrt{x}+1)\sqrt{4x^2-1} \\
=(1-\sqrt{x})(2\sqrt{2x+x^3}-(\sqrt{x}+1)\sqrt{4x^2-1}) \geq 0
\end{align*}
But we have,
\begin{align*}
2\sqrt{2x+x^3}-(\sqrt{x}+1)\sqrt{4x^2-1} \geq 2(\sqrt{2x+x^3}-\sqrt{4x^2-1})\\
=\frac{2({2x+x^3}-{4x^2+1})}{\sqrt{2x+x^3}+\sqrt{4x^2-1}}=\frac{2(1-x)(-x^2+3x+1)}{\sqrt{2x+x^3}+\sqrt{4x^2-1}} \geq 0
\end{align*}
since $1>x.$ Hence, the desired inequality is fulfilled.

Let us consider case 2). Then \eqref{INEQ} becomes
\begin{equation}\label{INEQSECONDISOS}
c  m_a+2\sqrt{ac} m_c \geq a m_a+2 c m_c 
\end{equation}
Putting $x=\frac{a}{c}<1,$ after division by $c^2,$ \eqref{INEQSECONDISOS} is equivalent to:
\[
\sqrt{4-x^2}+2\sqrt{x}\sqrt{1+2x^2} \geq x\sqrt{4-x^2}+2\sqrt{1+2x^2}
\]
Hence, 
\[
(1- x)\sqrt{4-x^2}-2\sqrt{1+2x^2}(1-\sqrt{x})\geq 0
\]
Simplifying,
\[
(1-\sqrt{x})\left((1+\sqrt{x})\sqrt{4-x^2}-2\sqrt{1+2x^2}\right)\geq 0
\]
Using the fact that $1+\sqrt{x} \geq 1+x,$ we have:
\begin{align*}
(1-\sqrt{x})\left((1+\sqrt{x})\sqrt{4-x^2}-2\sqrt{1+2x^2}\right)\\
\geq (1-\sqrt{x})\left((1+x)\sqrt{4-x^2}-2\sqrt{1+2x^2}\right)\\
=\frac{(1-\sqrt{x})\left((1+x)^2(4-x^2)-4(1+2x^2)\right)}{(1+x)\sqrt{4-x^2}+2\sqrt{1+2x^2}} \\
=\frac{(1-\sqrt{x})x\left(8-5x-2x^2-x^3\right)}{(1+x)\sqrt{4-x^2}+2\sqrt{1+2x^2}}  \geq 0,
\end{align*}
since $x<1.$
\end{proof}

Further we remind some facts we will use. For any non-degenerate triangle for which $a\leq b \leq c:$ 
\begin{equation}\label{abcmambmc}
m_a \geq m_b \geq m_c.
\end{equation}
\begin{equation}\label{triangle}
a+b >c 
\end{equation}
\begin{equation}\label{trianglemed}
m_c+m_b >m_a 
\end{equation}
\begin{lemma}\label{squareroot}
If $a,b,c$ are sides of a non-degenerate triangle, so are $\sqrt{a},\sqrt{b},\sqrt{c}.$
\end{lemma}
\begin{proof}Without loss of generality let $a \leq b \leq c.$ We have due to \eqref{triangle}
\[
\sqrt{a+b} > \sqrt{c}.
\]
It remains to prove that 
\begin{align*}
\sqrt{a}+\sqrt{b} > \sqrt{a+b}, \text{ i.e. }
\frac{\sqrt{a}}{\sqrt{a+b}}+\frac{\sqrt{b}}{\sqrt{a+b}} > 1.
\end{align*}
Since we have $\sqrt{x}>x$ for $0<x<1,$
\[
\sqrt{\frac{a}{a+b}}+\sqrt{\frac{b}{a+b}} > \frac{a}{a+b}+\frac{b}{a+b}=1.
\]
\end{proof}

\begin{lemma}\label{scalenetriangle}
If $a,b,c$ are sides of triangle and $a < b < c,$ we have 
\begin{equation}\label{mostimportant}
\sqrt{bc}m_a+\sqrt{ac}m_b \geq b m_b+ c m_c
\end{equation}
\end{lemma}
\begin{proof}
We will show that
\begin{equation}\label{twice}
2\sqrt{ac}m_b+2\sqrt{bc}m_a \geq a m_b +b m_a+ c m_a+a m_b 
\end{equation}
Due to \eqref{INEQKEY} we have
\[
a m_b +b m_a+ c m_a+a m_c \geq 2b m_b+2c m_c
\]
Hence, \eqref{twice} (since $m_b > m_c$) together with \eqref{INEQKEY} imply \eqref{mostimportant}.
We rewrite \eqref{twice} in the following manner:
\[
2(\sqrt{ac}-a)m_b \geq (c+b-2\sqrt{bc})m_a  \]
i.e.,
\[
(2m_b) \sqrt{a}(\sqrt{c}-\sqrt{a}) \geq m_a (\sqrt{c}-\sqrt{b})^2
\]
The last is obviously true since we have $2m_b>m_b+m_c >m_a$ and by Lemma~\ref{squareroot}: $\sqrt{a} > \sqrt{c} -\sqrt{b} $ and $\sqrt{c}-\sqrt{a} > \sqrt{c} - \sqrt{b},$ since $a<b.$
\end{proof}
 
\begin{lemma}\label{order} Let $a<b<c$ be sides of triangle and $m_a>m_b>m_c$ its medians. Then we have
\begin{equation}
B > \max(A,C),
\end{equation}
where $A=a m_a, B=b m_b, C=c m_c.$
\end{lemma}
\begin{proof}
First we will show that $B>C.$ 

We have to show that:
\[
\sqrt{2a^2b^2+2b^2c^2-b^4} -\sqrt{2b^2c^2+2a^2c^2-c^4} >0
\]
We obtain:
\begin{align*}
\sqrt{2a^2b^2+2b^2c^2-b^4} -\sqrt{2b^2c^2+2a^2c^2-c^4} \\
=\frac{2a^2b^2+2b^2c^2-b^4 -2b^2c^2-2a^2c^2+c^4}{\sqrt{2a^2b^2+2b^2c^2-b^4} +\sqrt{2b^2c^2+2a^2c^2-c^4}}\\
=\frac{(c^2-b^2)(c^2+b^2-2a^2)}{\sqrt{2a^2b^2+2b^2c^2-b^4} +\sqrt{2b^2c^2+2a^2c^2-c^4}}>0
\end{align*}

Next we will show that $B>A.$

We have to show that:
\[
\sqrt{2a^2b^2+2b^2c^2-b^4} -\sqrt{2a^2c^2+2a^2b^2-a^4} >0
\]
We obtain:
\begin{align*}
\sqrt{2a^2b^2+2b^2c^2-b^4} -\sqrt{2a^2c^2+2a^2b^2-a^4} \\
=\frac{2a^2b^2+2b^2c^2-b^4 -2a^2c^2-2a^2b^2+a^4}{\sqrt{2a^2b^2+2b^2c^2-b^4} +\sqrt{2a^2c^2+2a^2b^2-a^4}}\\
=\frac{(b^2-a^2)(2c^2-b^2 -a^2)}{\sqrt{2a^2b^2+2b^2c^2-b^4} +\sqrt{2a^2c^2+2a^2b^2-a^4}}>0
\end{align*}
Hence, $B>A$ and $B>C,$ i.e., $B>\max(A,C).$
\end{proof}

\section{Main result}

We are now ready to prove our main result.
\begin{theorem}\label{scalenetriag} For any non-degenerate triangle \eqref{INEQ} holds.
\end{theorem}
\begin{proof}
From Lemma~\ref{isosceles} we know the desired inequality is true for isosceles triangles. So without loss of generality we further assume that we have a scalene triangle with $a<b<c.$

We give two equivalent forms of \eqref{INEQ}, which we will use 
\begin{equation}\label{AgeqC}
\left(\frac{\sqrt{c}m_a}{\sqrt{b}m_b}-1\right)B+\left(\frac{\sqrt{c}m_b}{\sqrt{a}m_a}-1\right)A+\left(\frac{\sqrt{ab}}{c}-1\right)C \geq 0
\end{equation}
\begin{equation}\label{CgeqA}
\left(\frac{\sqrt{c}m_a}{\sqrt{b}m_b}-1\right)B+\left(\frac{\sqrt{a}m_b}{\sqrt{c}m_c}-1\right)C+\left(\frac{\sqrt{b}m_c}{\sqrt{a}m_a}-1\right)A \geq 0
\end{equation}

For any fixed triangle we have two possible cases 1) $A \geq C$ or 2)~$C >A.$

Let 1) be fulfilled. We choose \eqref{AgeqC}. By using the fact that $B \geq \max(A,C)$ (see Lemma~\ref{order}) to diminish the value of the expression and then applying the arithmetic mean - geometric mean inequality we obtain:
\begin{align*}
\underbrace{\left(\frac{\sqrt{c}m_a}{\sqrt{b}m_b}-1\right)}_{\geq 0} B+\left(\frac{\sqrt{c}m_b}{\sqrt{a}m_a}-1\right)A+\underbrace{\left(\frac{\sqrt{ab}}{c}-1\right)}_{\leq 0}C\\
\geq \left(\frac{\sqrt{c}m_a}{\sqrt{b}m_b}-1\right)A+\left(\frac{\sqrt{c}m_b}{\sqrt{a}m_a}-1\right)A+\left(\frac{\sqrt{ab}}{c}-1\right)A \\
= \left(\frac{\sqrt{c}m_a}{\sqrt{b}m_b}+\frac{\sqrt{c}m_b}{\sqrt{a}m_a}+ \frac{\sqrt{ab}}{{c} }    -3\right)A \geq 0
\end{align*}
Hence, \eqref{INEQ} is valid.

Let 2) be fulfilled. We choose \eqref{CgeqA}. We have two cases:

2)(i) $\frac{\sqrt{b}m_c}{\sqrt{a}m_a}-1 \leq 0.$ Then proceeding in a similar manner as above we obtain:
\begin{align*}
\underbrace{\left(\frac{\sqrt{c}m_a}{\sqrt{b}m_b}-1\right)}_{\geq 0} B+\left(\frac{\sqrt{a}m_b}{\sqrt{c}m_c}-1\right)C+\underbrace{\left(\frac{\sqrt{b}m_c}{\sqrt{a}m_a}-1\right)}_{\leq 0}A\\
\geq \left(\frac{\sqrt{c}m_a}{\sqrt{b}m_b}-1\right)C+\left(\frac{\sqrt{a}m_b}{\sqrt{c}m_c}-1\right)C+\left(\frac{\sqrt{b}m_c}{\sqrt{a}m_a}-1\right)C\\
= \left(\frac{\sqrt{c}m_a}{\sqrt{b}m_b} + \frac{\sqrt{a}m_b}{\sqrt{c}m_c} +\frac{\sqrt{b}m_c}{\sqrt{a}m_a}-3\right)C \geq 0
\end{align*}
Hence, \eqref{INEQ} is valid.

2)(ii) $\frac{\sqrt{b}m_c}{\sqrt{a}m_a}-1 > 0.$

Then we have 
\[
\left(\frac{\sqrt{b}m_c}{\sqrt{a}m_a}-1\right)A = \sqrt{ab}m_c-am_a >0
\]
But by Lemma~\ref{scalenetriangle} we have that \eqref{mostimportant} is fulfilled.
Hence, \eqref{INEQ} is fulfilled.
\end{proof}

\begin{corollary}\label{twovars} Let $0<x\leq y \leq 1$ and $x+y >1.$ Then the following inequality is always fulfilled:
\begin{align*}
\sqrt{y}\sqrt{2+2y^2-x^2}+\sqrt{x}\sqrt{2+2x^2-y^2}+\sqrt{xy}\sqrt{2x^2+2y^2-1}\\
\geq x\sqrt{2+2y^2-x^2}+y\sqrt{2+2x^2-y^2}+\sqrt{2x^2+2y^2-1} 
\end{align*}
\end{corollary}

\section{Additional results and discussion}

One may wonder if indeed \eqref{INEQ} and \eqref{INEQS} are only valid for the medians or they are true for other type of Cevians as well. We will show that replacing the medians with the respective altitudes or internal angle bisectors results in valid inequalities. In case of the altitudes, using the fact that:
\[
ah_a=bh_b=ch_c=2S
\]
we obtain that \eqref{INEQ} is equivalent to the AM-GM inequality, while \eqref{INEQS} is equivalent to:
\[
\frac{ab}{c}+\frac{ac}{b}+\frac{bc}{a}-a-b-c \geq 0
\]
Without loss of generality assume $a\leq b \leq c.$ Hence,
the above is equivalent to:
\[
\left(\frac{b}{a}-1\right)c+\left(\frac{ac}{b^2}-1\right)b+\left(\frac{b}{c}-1\right) a \geq 0
\]
But we have
\[
\underbrace{\left(\frac{b}{a}-1\right)}_{\geq 0}c+\left(\frac{ac}{b^2}-1\right)b+\underbrace{\left(\frac{b}{c}-1\right)}_{\leq 0} a \geq  b \left(\frac{b}{a}+\frac{ac}{b^2}+\frac{b}{c}-3\right) \geq 0
\]

Things are slightly more difficult with the internal angle bisectors. We will sketch the proof, but we will omit some of the tedious direct checks. 

\begin{lemma} For every triangle with sides $a,b,c$ such that $a \leq b \leq c,$ we have:
\[
bl_b \geq \max(al_a,cl_c)
\]
\end{lemma}
\begin{proof}
$ bl_b \geq al_a$ is equivalent to
\begin{equation}
\frac{\sqrt{b (a+c-b)} }{a+c}-         \frac{\sqrt{a (b+c-a)} }{b+c}  \geq 0                         
\end{equation}
This is true if
\begin{align*}
&\left(\frac{\sqrt{b (a+c-b)} }{a+c}-         \frac{\sqrt{a (b+c-a)} }{b+c}  \right)\left(\frac{\sqrt{b (a+c-b)} }{a+c}+        \frac{\sqrt{a (b+c-a)} }{b+c}  \right) \\
&=\frac{{b (a+c-b)} }{(a+c)^2}-    \frac{{a (b+c-a)} }{(b+c)^2}  =\frac{(b - a) (a + b + c) ( a b -a^2+c^2- b^2 )}{(a+c)^2(b+c)^2} \geq 0                      
\end{align*}
which is true since $a \leq b \leq c.$

$ bl_b \geq cl_c$ is equivalent to

\begin{equation}
\frac{\sqrt{b (a+c-b)} }{a+c}-         \frac{\sqrt{c (a+b-c)} }{a+b}  \geq 0                         
\end{equation}
This is true if
\begin{align*}
&\left(\frac{\sqrt{b (a+c-b)} }{a+c}-         \frac{\sqrt{c (a+b-c)} }{a+b}  \right)\left(\frac{\sqrt{b (a+c-b)} }{a+c}+        \frac{\sqrt{c (a+b-c)} }{a+b}   \right) \\
&=\frac{{b (a+c-b)} }{(a+c)^2}-    \frac{{c (a+b-c)} }{(a+b)^2}  =\frac{(c - b) (a + b + c) ( b^2 -a^2 +c^2- b c )}{(a+c)^2(a+b)^2} \geq 0                      
\end{align*}
which is true since $a \leq b \leq c.$

Hence, the condition of the lemma is fulfilled.
\end{proof}

Denoting everywhere below $al_a$ by $A,$ $bl_b$ by $B,$ and $cl_c$ by $C,$  and proceeding as in the case of the medians for the analogue of \eqref{INEQ} we obtain:

Case 1) $A \geq C$

Then 
\begin{align*}
\underbrace{\left(\sqrt{\frac{c}{b}}\frac{l_a}{l_b}-1\right)}_{\geq 0} B+\left(\sqrt{\frac{c}{a}}\frac{l_b}{l_a}-1\right)A+\underbrace{\left(\frac{\sqrt{ab}}{c}-1\right)}_{\leq 0}C\\
\geq \left(\sqrt{\frac{c}{b}}\frac{l_a}{l_b}+\sqrt{\frac{c}{a}}\frac{l_b}{l_a}+\frac{\sqrt{ab}}{c}-3\right)A \geq 0
\end{align*}

Case 2) $C > A$

Then, if $\sqrt{b}l_c \leq \sqrt{a}{l_a}$ we have
\begin{align*}
&\underbrace{\left(\sqrt{\frac{c}{b}}\frac{l_a}{l_b}-1\right)}_{\geq 0} B+\left(\sqrt{\frac{a}{c}}\frac{l_b}{l_c}-1\right)C+\underbrace{\left(\sqrt{\frac{b}{a}}\frac{l_c}{l_a}-1\right)}_{\leq 0} A\\
&\geq \left(\sqrt{\frac{c}{b}}\frac{l_a}{l_b}+\sqrt{\frac{a}{c}}\frac{l_b}{l_c}+\sqrt{\frac{b}{a}}\frac{l_c}{l_a}-3\right)C \geq 0
\end{align*}
However, it is easy to show that $\sqrt{b}l_c \leq \sqrt{c}l_c  \leq \sqrt{a}{l_a}.$

$\sqrt{c}l_c  \leq \sqrt{a}{l_a}$ is equivalent to
\[
0 \leq \frac{\sqrt{b+c-a}}{b+c}   -    \frac{\sqrt{a+b-c}}{a+b} , 
\]
But
\begin{align*}
&\frac{b+c-a}{(b+c)^2} -\frac{a+b-c}{(a+b)^2}=-\frac{1}{a+b}+   \frac{1}{b+c}     + \frac{c}{(a+b)^2} - \frac{a}{(b+c)^2} \\
& \geq  (c-a)\left(\frac{1}{(a+b)(b+c)}+\frac{1}{(a+b)^2}\right) \geq 0
\end{align*}
Hence,
\[
\frac{\sqrt{b+c-a}}{b+c}   -    \frac{\sqrt{a+b-c}}{a+b}=\dfrac{\dfrac{b+c-a}{(b+c)^2} -\dfrac{a+b-c}{(a+b)^2}}{\dfrac{\sqrt{b+c-a}}{b+c}   + \dfrac{\sqrt{a+b-c}}{a+b} } \geq 0
\]

For the analogue of \eqref{INEQS} we have

Case 1) $A \geq C$

Then 
\begin{align*}
& \underbrace{\left(c\frac{l_a}{l_b}-b\right)}_{\geq 0}B+\left(c\frac{l_b}{l_a}-a\right)A+\underbrace{\left(\frac{ab}{c}-c\right)}_{\leq 0}C\\
&\geq \left(c\frac{l_a}{l_b}-b\right)A+\left(c\frac{l_b}{l_a}-a\right)A+\left(\frac{ab}{c}-c\right)A\\
&=\left(\underbrace{\left(\frac{l_a}{l_b}-1\right)}_{\geq 0}c+\left(\frac{c}{b}\frac{l_b}{l_a}-1\right)b+\underbrace{\left(\frac{b}{c}-1\right)}_{\leq 0}a\right)A \\
&\geq \left(\frac{l_a}{l_b}+\frac{c}{b}\frac{l_b}{l_a}+\frac{b}{c}-3\right)bA \geq 0
\end{align*}

Case 2) $C > A$

Let $bl_c \leq a l_a \leq cl_c.$ Then,
\begin{align*}
& \underbrace{\left(c\frac{l_a}{l_b}-b\right)}_{> 0} B+\left(a\frac{l_b}{l_c}-c\right)C+\underbrace{\left(b\frac{l_c}{l_a}-a\right)}_{\leq 0} A  \geq \left( c\frac{l_a}{l_b}+a\frac{l_b}{l_c}+b\frac{l_c}{l_a}-a-b-c\right)C\\
&= \left(  \underbrace{\left( \frac{l_a}{l_b} -1\right)}_{\geq 0}c  + \left(\frac{a}{b}\frac{l_b}{l_c}            -1 \right)b+\underbrace{\left(\frac{b}{a}\frac{l_c}{l_a}-1 \right)}_{\leq 0}a  \right)C \geq \left( \frac{l_a}{l_b} +\frac{a}{b}\frac{l_b}{l_c}  +      \frac{b}{a}\frac{l_c}{l_a}-3 \right)  bC
\end{align*}
which is true. 

Let $bl_c > a l_a ,$ which is equivalent to $ab l_c \geq a^2 l_a.$ 
We will show that 
\[
bcl_a+acl_b \geq c^2l_c+b^2l_b
\]
After dividing both sides by $bcl_b,$ we obtain the equivalent inequality:
\[
\frac{l_a}{l_b}+\frac{a}{b} \geq \frac{C}{B}+\frac{b}{c}
\]
We will show that
\[
\frac{l_a}{l_b} \geq \frac{c+b-a}{c}
\]
which yields the desired inequality.

Indeed,
\begin{align*}
\frac{l_a}{l_b}= \frac{\sqrt{ b+c-a} }{\sqrt{ a+c-b} } \sqrt{\frac{b}{a}}\frac{a+c}{b+c} \geq \frac{\sqrt{ b+c-a} }{\sqrt{ a+c-b} }  = \frac{ b+c-a}{\sqrt{c^2-(b-a)^2}} \geq \frac{ b+c-a}{c} .
\end{align*}
Hence, the inequality also holds in this case.

\begin{corollary} We have:
\begin{align}
(\sqrt{bc}-a) f_a+(\sqrt{ac}-b) f_b+(\sqrt{ab}-c)  f_c \geq 0\\
(bc-a^2) f_a+(ac-b^2) f_b+(ab-c^2)  f_c \geq 0,
\end{align}
where:
\[
f_a=\alpha m_a+ \beta h_a + \gamma l_a; f_b=\alpha m_b+ \beta h_b + \gamma l_b ; f_c=\alpha m_c+ \beta h_c + \gamma l_c,
\]
for any $\alpha, \beta, \gamma \in [0,+\infty)$
\end{corollary}

It is not hard to construct a counterexample for the analogues of \eqref{INEQ} and \eqref{INEQS} pertaining to arbitrary Cevians. In particular, our counterexample did not satisfy the conditions that altitudes, medians and internal angle bisectors do satisfy, namely the analogues of \eqref{abcmambmc} and the fact that:
\[
B \geq \max(A, C).
\]

This observation naturally leads to the following question.
\begin{openproblem} Let $a \leq b \leq c$ be the sides of triangle. Then do there exist any Cevians  that satisfy both the inequalities:
\begin{equation}\label{eqcacabcc}
C_a \geq C_b \geq C_c
\end{equation}
and
\begin{equation}\label{BmaxAC}
b\, C_b \geq \max(a\, C_a, c\, C_c)
\end{equation}
but fail to satisfy either one or both of the following inequalities (analogues of \eqref{INEQ} and \eqref{INEQS}, respectively)?
\begin{equation*}
(\sqrt{bc}-a) C_a+(\sqrt{ac}-b) C_b+(\sqrt{ab}-c)  C_c \geq 0\tag{1$^*$} \label{EQ1prim}
\end{equation*}
\begin{equation*}
(bc-a^2) C_a+(ac-b^2) C_b+(ab-c^2)  C_c \geq 0.
\tag{2$^*$} \label{EQ2prim}
\end{equation*}
\end{openproblem}

\begin{remark} It seems likely that inequalities  \eqref{eqcacabcc} and \eqref{BmaxAC} are necessary but not sufficient for the fulfillment of \eqref{EQ1prim} and \eqref{EQ2prim}.

If we consider the splitters (the Cevians that pass through the Nagel point and bisect the perimeter of the triangle), for a triangle with sides $a\leq b \leq c,$ we have for their lengths:
\begin{align*}
n_a&=\sqrt{s}\sqrt{\frac{(c-b)^2}{a}+s-a};\\
n_b&=\sqrt{s}\sqrt{\frac{(c-a)^2}{b}+s-b};\\
n_c&=\sqrt{s}\sqrt{\frac{(b-a)^2}{c}+s-c},
\end{align*}
where $s=\frac{a+b+c}{2}$ is the semiperimeter.
It is not hard to show that $n_a-n_b \geq 0,$ and $n_b-n_c \geq 0.$ For the first consider the expression
\[
\frac{ab(n_a^2-n_b^2)}{s}=(b(c-b)^2+ab(b-a)-a(c-a)^2)=(b-a)(a+b-c)^2 \geq 0.
\]
Hence, $n_a-n_b \geq 0$ since $a,b,n_a+n_n$ and $s$ are all positive. Analogously, $n_b-n_c \geq 0$ is equivalent to
\[
\frac{bc(n_b^2-n_c^2)}{s}=(c(c-a)^2+bc(c-b)-b(b-a)^2)=(c-b)(c+b-a)^2 \geq 0.
\]
Therefore, \eqref{eqcacabcc} is fulfilled.

For \eqref{BmaxAC}, let us first establish that $B=b\, n_b \geq a\, n_a = A.$

We have to show that 
\[
b(c-a)^2+b^2(s-b) \geq a(c-b)^2+a^2(s-a)
\]
This inequality, however is equivalent to 
\[
(2c-a-b)(b-a)s \geq 0,
\]
which in view of $a\leq b \leq c$ is true. Thus, $B \geq A.$
For $B \geq  C= c\, n_c,$ we have to show
\[
b(c-a)^2+b^2(s-b) \geq c(b-a)^2+c^2(s-s),
\]
This in turn is equivalent to
\[
(c+b-2a)(c-b)s \geq 0.
\]
The last, in view of $a\leq b \leq c$ is obviously true.

Hence, \eqref{BmaxAC} is also fulfilled.

Let us consider the obtuse triangle with sides $a=4, b=16, c=18.$ For it we have,
$n_a=4\sqrt{19},n_b=\frac{1}{2}\sqrt{61}\sqrt{19}, n_c=3\sqrt{19}.$ A quick calculation shows that $C>A$ and both \eqref{EQ1prim} and  \eqref{EQ2prim} fail to hold in this case.
\end{remark}

\section{Conclusion}
We have provided a proof for an inequality which was conjectured to be true but had not been proved before. In our small quest for establishing the desired proof, we have stumbled upon similarly looking inequality and we have established that both inequalities may be extended to at least some other Cevians. We have thus given a new direction for investigation of the innumerable properties of triangles. Finally, we wish to point the interested reader to some wonderful books pertaining to inequalities, which have been an invaluable help and source of inspiration in our endeavour -- \cite{less_is_more,secrets,major}.


\begin{thebibliography}{99} %
\bibitem{Crux} Crux Mathematicorum, Vol.20(01), 1994, problem 1904, 16

\bibitem{Cruxsol} Crux Mathematicorum, Vol.20(10), 1994, problem 1904, 289-290
\bibitem{MathematicalReflections} Mathematical Reflections, Vol. 6, 2016, problem J394

\bibitem{cuttheknot} \href{https://www.cut-the-knot.org/triangle/InequalityWithSidesAndMedians.shtml}{Bogomolny, A. (n.d.). An Inequality with Sides and Medians. Interactive Mathematics Miscellany and Puzzles. (last accessed on 10 Feb 2026)} 

\bibitem{less_is_more} Alsina, C.,  R. B. Nelsen. When less is more: visualizing basic inequalities, Mathematical Association of America, Washington, 2009.

\bibitem{secrets} Hung, P. K. Secrets in inequalities. GIL Publishing House, Vol.~1, 2008.


\bibitem{major} Marshall, A. W., I. Olkin, B. Arnold. Inequalities: Theory of Majorization and Its Applications, Second edition. Springer Series in Statistics. Springer, New York, 2011. 

\end{thebibliography}
\end{document}